\documentclass[letterpaper, 10 pt, conference]{ieeeconf}  

\IEEEoverridecommandlockouts                              

\usepackage{graphics} 
\usepackage{epsfig} 
\usepackage{mathptmx} 
\usepackage{times} 
\usepackage{amsmath} 
\usepackage{amssymb}  
\usepackage{amsmath,amsfonts,mathrsfs,latexsym,amscd}
\usepackage{mathtools}
\usepackage{physics}
\usepackage[ruled]{algorithm2e}
\usepackage{graphicx}
\usepackage[dvipsnames,svgnames]{xcolor}
\usepackage{hyperref}

\usepackage[dvipsnames,svgnames]{xcolor}
\colorlet{MyBlue}{DodgerBlue!75!Black}
\colorlet{MyGreen}{DarkGreen!95!Black}

\newcommand{\eps}{\varepsilon}
\DeclareMathOperator*{\argmin}{argmin}

\DeclareMathOperator{\zer}{zer}

\DeclareMathOperator{\dom}{dom}

\DeclareMathOperator{\gr}{graph}

\DeclareMathOperator{\resolvent}{\mathsf{J}}

\DeclareMathOperator{\Id}{Id}

\newcommand{\ce}{\mathtt{e}}

\DeclareMathOperator{\VI}{VI}
\DeclareMathOperator{\NE}{NE}

\newcommand{\bT}{\mathbf{T}}

\renewcommand{\iff}{\Leftrightarrow}

\renewcommand{\emptyset}{\varnothing}
\newcommand{\eqdef}{\triangleq}
\newcommand{\wlim}{\rightharpoonup}

\newcommand{\scrG}{\mathcal{G}}
\newcommand{\scrH}{\mathcal{H}}

\newcommand{\scrK}{\mathcal{K}}

\newcommand{\scrS}{\mathcal{S}}

\newcommand{\scrX}{\mathcal{X}}

\newcommand{\R}{\mathbb{R}}

\newcommand{\N}{\mathbb{N}}

\newcommand{\K}{\mathbb{K}}

\DeclareMathOperator{\NC}{\mathsf{N}}

\newcommand{\Lip}{L}

\newcommand{\opA}{\mathbb{A}}

\newcommand{\opG}{\mathbb{G}}

\newcommand{\opF}{\mathbb{F}}
\newcommand{\FB}{\mathtt{FB}}

\newtheorem{theorem}{Theorem}

\newtheorem{lemma}[theorem]{Lemma}

\newtheorem{assumption}{Assumption}

\newtheorem{example}{Example}

\DeclarePairedDelimiter{\inner}{\langle}{\rangle}

\title{\LARGE \bf
Tikhonov regularized exterior penalty methods for hierarchical variational inequalities
}

\author{Meggie Marschner and Mathias Staudigl$^{1}$
\thanks{This research benefited from the support of the FMJH Program Gaspard Monge for optimization and operations research and their interactions with data science. MST research is sponsored by the Deutsche Forschungsgemeinschaft (DFG) - Projektnummer 556222748 ("Non-stationary hierarchical optimization")  }
\thanks{$^{1}$ Mannheim University, Department of Mathematics, B6 26, 68159 Mannheim
        {\tt\small m.staudigl@uni-mannheim},  {\tt\small m.marschner@uni-mannheim.de}}}%
\begin{document}

\maketitle
\thispagestyle{empty}
\pagestyle{empty}

\begin{abstract}

\noindent We consider nested variational inequalities consisting in a (upper-level) variational inequality whose feasible set is given by the solution set of another (lower-level) variational inequality. This class of hierarchical equilibrium contains a wealth of important applications, including purely hierarchical convex bilevel optimization problems and certain multi-follower games. Working within a real Hilbert space setting, we develop a double loop prox-penalization algorithm with strong convergence guarantees towards a solution of the nested VI problem. We present various application that fit into our framework and present also some preliminary numerical results. 
\end{abstract}

\section{INTRODUCTION}

\noindent In a real Hilbert space setting, we consider the hierarchical variational inequality problem 
\begin{equation}\label{eq:P}\tag{P}
\VI(\opG,\scrS_{0}),\quad\text{where }\scrS_{0}:=\zer(\opA+\opF) 
\end{equation}
in which the problem data satisfy the following conditions: 
\begin{assumption}\label{ass:1}
\begin{enumerate}
\item $\opG:\scrH\to\scrH$ is a monotone and Lipschitz continuous map with Lipschitz constant $L_{G}$ and weakly sequentially continuous, i.e. $v^{n}\wlim v\Rightarrow \opG(v^{n})\wlim \opG(v)$; 
\item $\opF:\scrH\to\scrH$ is monotone and Lipschitz with Lipschitz constant $L_{F}$ and weakly sequentially continuous;  
\item $\opA:\scrH\to 2^{\scrH}$ is maximally monotone with $\dom(\opA)$ bounded;
\item The set $\scrS_{0}$ is nonempty. 
\end{enumerate}
\end{assumption}
Our assumptions imply that $\dom(\opF)=\scrH$, and thus $\opA+\opF$ is maximally monotone. Furthermore, it follows that $\scrS_{0}$ is closed and convex. The problem $\VI(\opG,\scrS_{0})$ is to find a point $\bar{u}\in\scrH$ such that 
\begin{equation}\label{eq:Sol}
\inner{\opG(\bar{u}),u-\bar{u}}\geq 0 \qquad\forall u\in\scrS_{0}=\zer(\opA+\opF).
\end{equation}
This nested VI problem encompasses a wide range of optimization and equilibrium problems involving a hierarchical structure. Concrete applications of this model framework can be found in signal processing \cite{Facchinei:2014aa}, equilibrium selection in Nash games \cite{Benenati:2022aa,Benenati:2023aa,Kaushik:2021aa}, inverse problems \cite{Yamada2011MinimizingTM}, certain classes of bilevel optimization \cite{Solodov:2007aa,yousefian2021bilevel}, and power allocation \cite{Pang:2008aa}, to mention a few. We give below a few concrete examples.

\begin{example}[Simple Bilevel Optimization]
\label{ex:SBP}
The seminal references \cite{Solodov:2007aa,dempe2010optimality,Dempe:2021aa} introduced simple bilevel problems of the form 
\begin{equation}\label{eq:SBP}
\min g(u) \quad\text{s.t.: } u\in\argmin_{v\in\scrH}\{r(v)+f(v)\} .
\end{equation}
Here $g,f:\scrH\to\R$ are convex Fr\'{e}chet-differentiable functions with Lipschitz continuous gradient $\opG=\nabla g$ and $\opF=\nabla f$, respectively. The function $r:\scrH\to \R\cup\{+\infty\}$ is a convex, proper and lower semicontinuous function, giving rise to the maximally monotone operator $\opA=\partial r$. Under the convexity assumptions, we can reformulate the problem \eqref{eq:SBP} as the hierarchical VI \eqref{eq:P}. More generally, we can consider structured convex optimization problems of the form 
$$
\min_{u}g(u)\quad\text{s.t. } u\in\argmin_{v\in\scrH}\{f(v)+r(Lv)\}.
$$
Using the Fenchel-Rockafellar duality in the lower level problem, we can reformulate the above as the saddle-point problem
\[
\min_{v}\{f(v)+r(Lv)\}=\min_{v}\max_{w}\{f(v)+\inner{L^{\ast}v,w}-r^{\ast}(w)\}.
\]
A saddle point is characterized by the monotone inclusion 
\begin{align*}
0\in \nabla f(\bar{v})+L\bar{w} \\
0\in -L^{\ast}\bar{v}+\partial r^{\ast}(\bar{w})
\end{align*}
which can be compactly written as the problem of finding a point in $\zer(\opA+\opF)$, with $\opF(v,w)=[Lw;-L^{\ast}v]+[\nabla f(v);0]$ and $\opA(v,w)=\{0\}\times\partial r^{\ast}(w)$. Note that the operator $\opF$ is not cocoercive, which precludes a direct approach via viscosity techniques as the celebrated BiG-SAM \cite{SabSht17}. 
\end{example}

\begin{example}[Equilibrium Selection in Nash games]
The Nash equilibrium problem is described by a finite set of players $i\in\{1,\ldots,N\}$, each characterized by a strategy space $\scrX_{i}\subseteq\scrH_{i}$, and a real-valued function $f_{i}:\prod_{i=1}^{N}\scrX_{i}\to \R$. An $N$-tuple $x^{\ast}=(x_{1}^{\ast},\ldots,x_{N}^{\ast})\in\scrX\eqdef\prod_{i=1}^{N}\scrX_{i}$ is a Nash equilibrium if 
$$
f_{i}(x^{\ast})\leq f_{i}(x_{i},x^{\ast}_{-i})\qquad \forall x_{i}\in\scrX_{i}, i=1,2,\ldots,N.
$$
We define the game as a tuple $\Gamma=\{f_{i},\scrX_{i}\}_{1\leq i\leq N}$ and denote the set of Nash equilibrium points by $\NE(\Gamma)$. Assuming that each mapping $x_{i}\mapsto f_{i}(x_{i},x_{-i})$ is convex and differentiable, it is a classical result that the set of Nash equilibria can be characterized as the solution set of a Variational inequality $\VI(\opF,\scrX)$, determined by the operator 
$$
\opF:\scrH\to\scrH,x\mapsto \opF(x)=\begin{pmatrix} \nabla_{x_{1}}f_{1}(x) \\ \vdots \\ \nabla_{x_{N}}f_{N}(x)\end{pmatrix}.
$$
Setting $\opA=\NC_{\scrX}$, the Normal cone operator (see section \ref{prelim}), we can reformulate problem $\VI(\opF,\scrX)$ as the inclusion problem $\zer(\opA+\opF)$. In many design problems, one asks for the "best" equilibrium point, relative to a pre-defined loss function $g:\scrH\to\R$ with gradient mapping $\opG=\nabla g$. This leads to the equilibrium selection problem 
$$
\text{Solve }\VI(\opG,\scrS_{0}) \quad\scrS_{0}=\zer(\opA+\opF).
$$
\end{example}

\subsection{Contributions}

\noindent In order to resolve the hierarchical VI problem \eqref{eq:P}, we follow \cite{Facchinei:2014aa} and adopt a regularization framework involving a parametric family of structured monotone inclusions $\opA+\Phi_{\alpha,\beta}(\bullet,w)$, where $\Phi_{\alpha,\beta}:\scrH\times\scrH\to\scrH$ is defined by  
\begin{equation}\label{eq:Phi}
\Phi_{\alpha,\beta}(v,w):=\opF(v)+\beta\opG(v)+\alpha(v-w).
\end{equation}
The parameter $\alpha>0$ is a proximal parameter, while $\beta>0$ is a Tikhonov parameter that regulates the relative importance of the lower level versus the upper level. Auxiliary mappings of the form \eqref{eq:Phi} have a long tradition in mathematical optimization in connection with the proximal penalization framework \cite{Alart:1991aa,Bahraoui:1994aa,Cominetti:1997aa}. To motivate this construction, let us investigate the structure of the operator $\Phi_{\alpha,\beta}$ when applied to the simple bilevel optimization problem (Example \ref{ex:SBP}). To iteratively solve this problem, suppose that $w$ is our best-so-far solution candidate. To obtain a new proposal, a common strategy within the proximal penalization framework is to solve the following strongly convex auxiliary problem 
\[
\min_{v\in\scrH}\{H_{\alpha,\beta}(v,w)\eqdef r(v)+f(v)+\beta g(v)+\frac{\alpha}{2}\norm{v-w}^{2}\}.
\]
Thanks to the proximal term, this problem admits a unique solution $\bar{u}_{\alpha,\beta}(w)$. Under a constraint qualification, we can compute 
$\partial_{v} H_{\alpha,\beta}(v,w)=\partial r(v)+\nabla f(v)+\beta\nabla g(v)+\alpha (v-w)=\opA(v)+\Phi_{\alpha,\beta}(v,w).$ 

Given an anchor point $w\in\scrH$, the auxiliary problem $\zer(\opA+\Phi_{\alpha,\beta}(\bullet,w))$ admits a unique solution $\bar{u}_{\alpha,\beta}(w)$. Our algorithmic framework builds upon a double loop architecture, in which the inner loop is an iterative method that drives the process to a neighborhood of the temporary solution $\bar{u}_{\alpha,\beta}(w)$. Upon termination of the inner loop, an outer loop is activated which updates the anchor point $w$ and the Tikhonov parameter $\beta$. The combination of Tikhonov and proximal regularization terms is borrowed from the paper \cite{Facchinei:2014aa}. This paper develops a \emph{general template} to design iterative methods for resolving problem \eqref{eq:P}. To keep the exposition concrete, we illustrate our machinery only in the special case in which the forward-backward splitting method is used as the main algorithmic map. The general abstract setting will be described in a forthcoming publication.  

With respect to \cite{Facchinei:2014aa}, the innovations of this note are the following: We consider a more general class of splitting problems, with explicit iterations via a resolvent step. Second, our scheme is robust to inexact computations and is formulated in Hilbert spaces, as required for potentially infinite-dimensional control applications. Our inner loop scheme is a relaxed-inertial forward backward implementation involving the regularized monotone operators that we employ in order to iteratively track the temporal solutions. This algorithmic technique is a classical tool in monotone splitting approaches. In fact, the inertial-relaxation version has been first studied in a discrete-time algorithm in \cite{Lorenz:2015aa}. However, their analysis assumes that the discrete velocity of the process is summable. This is a rather frequently encountered technical assumption, which, however, can never be checked before the algorithm is actually executed. Our proof does not make such an assumption, and instead uses a new Lyapunov-type analysis, inspired by \cite{Cortild:2024aa}. Using this technique, we establish linear convergence of the inner loop scheme. Additionally, we develop convergence guarantees in an inexact computational framework. Inexact forward-backward algorithms are studied in \cite{villa2013accelerated} and \cite{Apidopoulos:2020aa}.

Hybrid models for approaching problem \eqref{eq:P} have been studied in \cite{Moudafi:2006aa,Mainge:2007aa,Marino:2011aa}. These methods require $\opF$  to be co-coercive (inverse strongly monotone), a condition that substantially holds when the lower-level VI reduces to a convex optimization problem. On the contrary, here we only require $\opF$ to be monotone and Lipschitz continuous, allowing for the treatment of more general problems with respect to these classical approaches.


\section{Preliminaries}\label{prelim}

\noindent Let $\scrH$ be a real Hilbert space with inner product $\inner{\cdot,\cdot}$ and corresponding norm $\norm{\cdot}$. Given a closed convex set $\scrK\subseteq\scrH$, we define the orthogonal projection operator $\Pi_{\scrK}(x)=\argmin\{\frac{1}{2}\norm{w-x}^{2}\vert w\in\scrK\}$.  A set-valued operator $\opA:\scrH\to 2^{\scrH}$ is monotone if 
\[
\inner{u-v,b-a}\geq 0\quad\forall (u,b)\in\gr(\opA),\forall (v,a)\in\gr(\opA), 
\]
where $\gr(\opA)\eqdef\{(v,a)\in\scrH\times\scrH\vert a\in\opA(v)\}$ is the graph of the operator $\opA$. The operator $\opA:\scrH\to 2^{\scrH}$ is maximally monotone if it is monotone and there exists no other monotone operator whose graph contains $\gr(\opA)$. The resolvent of an operator $\opA$ is defined as $\resolvent_{\opA}=(\Id+\opA)^{-1}$. If $\opA$ is maximally monotone, then the resolvent is a nonexpansive operator (and thus single-valued) \cite{BauCom16}. 
\begin{lemma}\label{lem:graphMM}
Let $\opA$ be a maximal monotone operator. A point $(x,a)\in\scrH\times\scrH$ belongs to $\gr(\opA)$ if and only if 
\[
\inner{a-b,x-y}\geq 0\qquad\forall (y,b)\in\gr(\opA).
\]
\end{lemma}
This Lemma shows also that for all $u\in\dom(\opA)\eqdef\{u\in\scrH\vert \opA(u)\neq\emptyset\}$, the set $\opA(u)$ is closed and convex. 

Given a mapping $\opG:\scrH\to\scrH$ and a set $\scrK\subseteq\scrH$, the variational inequality problem $\VI(\opG,\scrK)$, is the problem of finding a point $u\in\scrK$ such that 
$
\inner{\opG(u),v-u}\geq0\;\;\forall v\in\scrK. 
$
The solution set of $\VI(\opG,\scrK)$ can be expressed as $\zer(\opG+\NC_{\scrK})$, involving the normal cone operator 
\[
\NC_{\scrK}(v):=\left\{\begin{array}{ll} 
\emptyset & \text{if }v\notin\scrK\\
\{p\in\scrH\vert \inner{p,u-v}\leq 0,\;\forall u\in\scrK\} & \text{ if }v\in\scrK.
\end{array}\right.
\]
If $\opG:\scrH\to\scrH$ is continuous and monotone, and $\scrK\subseteq\scrH$ is closed and convex, then $\zer(\opG+\NC_{\scrK})$ is convex.
\begin{lemma}\label{lem:solution}
Let $\opG:\scrH\to\scrH$ be a monotone operator and $\scrK\subseteq\scrH$ convex. Let $u^{\ast}\in\zer(\opG+\NC_{\scrK})$. If $v\in\scrK$ satisfies
$
\inner{\opG(v),u^{\ast}-v}\geq 0,
$
 then $v\in\zer(\opG+\NC_{\scrK})$.
\end{lemma}

\begin{lemma}\label{lem1}
For all $u,v\in\scrH$ and $\xi\in\R$ it holds that 
$$
\norm{u+\xi v}^2=(1+\xi)\norm{u}^2+\xi(1+\xi)\norm{v}^2-\xi\norm{u-v}.
$$
\end{lemma}

\begin{lemma}\label{lem2}
For all $u,v\in\scrH$ and $t>0$ holds
\begin{multline*}
(1-t)\norm{u}^2+(1-\frac{1}{t})\norm{v}^2\\\leq\norm{u\pm v}^2\leq(1+t)\norm{u}^2+(1+\frac{1}{t})\norm{v}^2.
\end{multline*}
\end{lemma}
The next important result is implicit in Theorem 2.1 of \cite{Alvarez:2001aa}. 

\begin{lemma}\label{lem3}
Let $(a_k)_{k\in\N}$ and $(b_k)_{k\in\N}$ be nonnegative sequences such that $a_k\leq a<1$ for all $k\geq 0$, and $\sum_{k=0}^{\infty}b_k<+\infty$. Consider a real sequence $(C_k)_{k\in\N}$ such that 
$$
C_{k+1}\leq a_kC_k+b_k
$$
for all $k\geq 0$. Then $\sum_{k=1}^{\infty}[C_k]_+$ is convergent. If $b_k\equiv 0$, then $[C_k]_+\leq a^{k-1}[C_1]_+$. Either way, if $C_k\geq w_k-dw_{k-1}$ for all $k\geq 0$ with $d\in[0,1]$ and $(w_k)_{k\in\N}$ nonnegative, then $(w_k)$ is convergent. If $d<1$, then $\sum_{k=0}^{\infty}w_k<+\infty$.
\end{lemma}

In the next results, we gather some basic properties of the operator $\Phi_{\alpha,\beta}(\cdot,w):\scrH\to\scrH$. 

\begin{lemma}\label{lem:Properties-Phi}
For all $(\alpha,\beta)\in\R^{2}_{+}$ the mapping $\Phi_{\alpha,\beta}(\cdot,w)$ is $\alpha$-strongly monotone and $\Lip_{\alpha,\beta}=\Lip_{\opF}+\beta \Lip_{\opG}+\alpha$ - Lipschitz continuous.
\end{lemma}

\noindent Following the classical literature on proximal diagonal schemes, we consider the sequence of auxiliary problems $\zer(\opA+\Phi_{\alpha,\beta}(\cdot,w))$. Denote by $\bar{u}_{\alpha,\beta}(w)$ the unique solution of the monotone inclusion problem $\zer(\opA+\Phi_{\alpha,\beta})$. Associated with this monotone inclusion problem, we have the equivalent fixed point reformulation using the forward-backward operator. Specifically, for $\gamma>0$, we define the merit function
\begin{equation}
\scrG^{\gamma}_{\alpha,\beta}(v,w):=\frac{1}{\gamma}[v-\resolvent_{\gamma\opA}(v-\gamma\Phi_{\alpha,\beta}(v,w))].
\end{equation}

This function has the property that 
$$
\scrG^{\gamma}_{\alpha,\beta}(v,w)=0\iff v=\bar{u}_{\alpha,\beta}(w).
$$
Hence, we make use of the norm $\norm{\scrG^{\gamma}_{\alpha,\beta}(v,w)}$ to measure the accuracy of the point $v$ with respect to the unique element of $\zer(\opA+\Phi_{\alpha,\beta}(\cdot,w))$, a strategy that achieves its formal justification via the next result. 

\begin{lemma}\label{lem:gradient}
For all $\alpha,\beta>0$ and $(v,w)\in\scrH\times\scrH$, we have 
\begin{equation}
\norm{v-\bar{u}_{\alpha,\beta}(w)}\leq\frac{1+\gamma\Lip_{\alpha,\beta}}{\alpha}\norm{\scrG^{\gamma}_{\alpha,\beta}(v,w)}.
\end{equation}
\end{lemma}
\begin{proof}
Call $g=\scrG^{\gamma}_{\alpha,\beta}(v,w)$, so that 
$$
\opA(v-g\gamma)\ni g-\Phi_{\alpha,\beta}(v,w).
$$
Additionally, $\bar{u}\equiv\bar{u}_{\alpha,\beta}(w)$ satisfies 
$$
-\Phi_{\alpha,\beta}(\bar{u},w)\in\opA(\bar{u}). 
$$
From the monotonicity of $\opA$, it follows  
$$
\inner{v-g\gamma-\bar{u},g-\Phi_{\alpha,\beta}(v,w)+\Phi_{\alpha,\beta}(\bar{u},w)}\geq 0.
$$
Rearranging the above immediately yields the estimate 
\begin{align*}
&\inner{v-\bar{u},\Phi_{\alpha,\beta}(v,w)-\Phi_{\alpha,\beta}(\bar{u},w)}\\
&\leq -\gamma\norm{g}^{2}+\inner{v-\bar{u},g}+\gamma\inner{g,\Phi_{\alpha,\beta}(v,w)-\Phi_{\alpha,\beta}(\bar{u},w)}\\
&\leq \norm{v-\bar{u}}\cdot\norm{g}+\gamma\norm{g}\cdot\norm{\Phi_{\alpha,\beta}(v,w)-\Phi_{\alpha,\beta}(\bar{u},w)}\\
&\leq \norm{v-\bar{u}}\cdot\norm{g}+\gamma L_{\alpha,\beta}\norm{v-\bar{u}}.
\end{align*}
Finally, using the $\alpha$-strong monotonicity of the operator $\Phi_{\alpha,\beta}(\cdot,w)$, uniformly in $w$, we can lower bound the left-hand side expression by 
\[
\alpha\norm{v-\bar{u}}^{2}\leq  \norm{v-\bar{u}}\cdot\norm{g}+\gamma\Lip_{\alpha,\beta}\norm{v-\bar{u}}.
\]
We conclude $\norm{v-\bar{u}}\leq \frac{1+\gamma\Lip_{\alpha,\beta}}{\alpha}\norm{g}.$
 \end{proof}

\section{Algorithm}
\noindent As mentioned, our algorithmic design uses a double loop architecture, in which a forward-backward splitting method serves as an inner loop scheme to produce an iterate that is close to a temporal solution of the auxiliary problem $\zer(\opA+\Phi_{\alpha,\beta}(\cdot,w))$.  We denote by $t\in\N$ the outer loop iteration count, and the corresponding time $t$ approximate solution $w^{t}\in\scrH$ which is produced in such a way that guarantees 
\begin{equation}\label{eq:temporal}
\norm{w^{t+1}-\bar{u}_{\alpha,\beta_{t}}(w^{t})}\leq\ce_{t},
\end{equation}
where $(\ce_{t})_{t\in\N}$ is a positive sequence of accuracies that is dynamically adjusted when the outer loop is called upon.  

\subsection{The inner loop}
\noindent
The main iteration performed by our method is a forward backward step, involving a set of user-provided parameters that are updated over time. We call this scheme by $\FB(\beta,\eps,\gamma,w)$. The main computational step in this procedure is the recursive update 
\begin{equation}\label{eq:FB}
v^{k+1}=T_{\gamma_k}(v^{k})\equiv \resolvent_{\gamma_{k}\opA}(v^{k}-\gamma_{k}\Phi_{\alpha,\beta}(v^{k},w)),
\end{equation}
where $w\in\scrH$ is a given anchor point and $(\gamma_{k})_{k}$ is a positive sequence of step sizes. Given the pair $(\alpha,\beta)\in(0,\infty)^{2}$ and $w\in\scrH$, our method recursively constructs sequences $(z^{k})_{k}$ and $(v^{k})_{k}$ which iteratively approximate the unique solution $\bar{u}_{\alpha,\beta}(w)$ of the auxiliary problem $\zer(\opA+\Phi_{\alpha,\beta}(\cdot,w))$. In the updating step of the sequence $(v^{k})_{k}$ instead of using the exact forward-backward map $T_{\gamma_k}(v)$, we assume that we have only access to a $\delta_{k}$-perturbation $\tilde{T}_{k}(z^{k})$, in the sense that 
\begin{equation}\label{eq:delta}
\tilde{T}_{k}(z)=T_{\gamma_{k}}(z)+\delta_{k}.
\end{equation}
The inclusion of the numerical error $\delta_{k}$ is motivated by the fact that the resolvent is in general not available in exact form or its computation may be very demanding. Just to mention some examples, this happens when applying proximal methods to image deblurring with total variation \cite{Chambolle:2004aa}, or to structured sparsity regularization problems in machine learning and inverse problems \cite{Zhao:2009aa}. In those cases, the proximity operator is usually computed using ad hoc algorithms, and therefore inexactly. Employing this inexact computational model, we propose an inertial forward-backward algorithm for iteratively solving \eqref{eq:P} by updating the sequence $(v^{k},z^{k})_{k}$ as follows: 
\[
\left\{\begin{array}{l}
z^{k}=v^{k}+\tau_{k}(v^{k}-v^{k-1}),\\
v^{k+1}=(1-\theta_{k})z^{k}+\theta_{k}\tilde{T}_{k}(z^{k})
\end{array}
\right.
\]
$\tau_{k}\in(0,1)$ is a momentum term, while the parameter $\theta_{k}\in(0,1)$ is a relaxation factor for the inexact Krasnoselskii-Mann iteration. It is well known that if we choose $\gamma_k\in(0,2\alpha/\Lip^{2}_{\alpha,\beta_t})$ and setting $q_k(\alpha,\beta_t)=\sqrt{1-\gamma_k(2\alpha-\gamma\Lip^{2}_{\alpha,\beta_t})}\in(0,1)$, then the operator $T_{\gamma_k}$  is a $q_k(\alpha,\beta_t)$ contraction.

As a last ingredient in the lower level problem, we need an operational stopping criterion which informs us when the iterates are sufficiently close to the temporal solution $\bar{u}_{\alpha,\beta}(w)$. To achieve this aim, we define the stopping time 
\begin{equation}\label{eq:stopping}
\K_{\alpha,\beta}(\eps,w)\eqdef\inf\{k\geq 1\vert\;\norm{v^{k+1}-z^{k}}\leq\eps\} .
\end{equation}

\noindent We introduce the following assumptions on the parameters:
\begin{assumption}\label{ass:slowcontrol}
$(\beta_{t})_{t\in\N}\notin\ell^{1}_{+}(\N)$ and $\lim_{t\to\infty}\frac{\ce_{t}}{\beta_{t}}=0$. 
\end{assumption}
\begin{assumption}\label{ass:parameters}
$\theta_k\in[0,1]$ and there exists a $\bar{\theta},\underline{\theta}$ such that $\bar{\theta}\geq\theta_k\geq\underline{\theta}$ for all k and  $\tau_k\in[0,1]$ monotonically increasing such that there exists a $\bar{\tau}\geq\tau_k$ for all k.
\end{assumption}

\subsection{General Stopping Time}
\noindent Before tackling the outer loop, we need an operational stopping criterion
which informs us when the iterates are sufficiently close to the temporal solution $\bar{u}_{\alpha,\beta_t}$. Let there exists an $\epsilon$ such that $\norm{v^{k+1}-z^{k}}\leq\epsilon$, then
\begin{align*}
\norm{z^k-T_{\gamma_k}(z^k)}&=\norm{\frac{1}{\theta_k}(v^{k+1}-z^k)}\leq\frac{\epsilon}{\theta_k}.
\end{align*}
Moreover, let again $\{\bar{u}^k\}=\textnormal{Fix}(T_{\gamma_k})$, we get 
\begin{align*}
\norm{z^k-\bar{u}^k}&=\norm{z^k-T_{\gamma_k}(z^k)+T_{\gamma_k}(z^k)-\bar{u}^k}\\
&=\norm{\frac{v^{k+1}-z^k}{\theta_k}+T_{\gamma_k}(z^k)-\bar{u}^k}\\
&\leq\norm{\frac{v^{k+1}-z^k}{\theta_k}}+\norm{T_{\gamma_k}(z^k)-T_{\gamma_k}(\bar{u}^k)}\\
&\leq\frac{\epsilon}{\theta_k}+q_k(\alpha,\beta_t)\norm{z^k-\bar{u}^k}
\end{align*}
so that $\norm{z^k-\bar{u}^k}\leq\frac{\epsilon}{\theta_k(1-q_k(\alpha,\beta_t))}.$ This suggests to take as a sensible stopping criterion the stopping time
$$
\mathbb{K}(\epsilon)=\inf\{k\geq 1\mid \norm{v^{k+1}-z^k}\leq\epsilon\}.
$$

\subsection{The outer loop} 
\noindent In the outer loop we update the parameters $\eps_{t},\beta_{t},\gamma_{t},w^{t}$ in order to restart the inner loop of the method $IKM(\alpha,\beta,\eps,\gamma,w)$. We set $K_{t}\equiv \K_{\alpha,\beta_{t}}(\eps_{t},w^{t})$, and update the anchor point by setting
$w^{t+1}=v^{K_{t}}_{t}$, so that condition \eqref{eq:temporal} holds with accuracy $\eps_{t}$. Within this loop, we assume that the step size $\gamma_{k}$ is fixed to $\gamma_{t}\in(0,\frac{2\alpha}{\Lip_{\alpha,\beta_{t}}})$. Indeed, by using Lemma \ref{lem:gradient}, we observe for $k\geq 1$:
\begin{align*}
\norm{z^{k}-\bar{u}_{\alpha,\beta_{t}}(w^{t})}&\leq \frac{1+\gamma_{t}\Lip_{\alpha,\beta_{t}}}{\alpha}\norm{\scrG^{\gamma_{t}}_{\alpha,\beta_{t}}(z^{k}_{t},w^{t})}\\
&=\frac{1+\gamma_{t}\Lip_{\alpha,\beta_{t}}}{\alpha\gamma_{t}}\norm{z^{k}-T_{\gamma_{t}}(z^{k})}\\
&= \frac{1+\gamma_{t}\Lip_{\alpha,\beta_{t}}}{\alpha\gamma_{t}}\norm{\frac{v^{k+1}-z^{k}}{\theta_{k}}+\delta_{k}}\\ 
&\leq \frac{1+\gamma_{t}\Lip_{\alpha,\beta_{t}}}{\alpha\theta_{k}\gamma_{t}}\left(\norm{v^{k+1}-z^{k}}+\theta_{k}\norm{\delta_{k}}\right).
\end{align*}
As an illustrative example, let us assume that error model is given by $\norm{\delta_{k}}=d_{k}\eps_{t}$ for a sequence $(d_{k})_{k\in\N}$. Then, there exists $D>0$ such that $(d_{k})\subseteq[0,D]$, and thus we can evaluate the previous inequality at the last iteration counter of the inner loop $k=\K_{t}\equiv \K_{\alpha,\beta_{t}}(\eps_{t},w^{t})$, to obtain
\begin{align}
\norm{w^{t+1}-\bar{u}_{\alpha,\beta_{t}}(w^{t})}&\leq \frac{1+\gamma_{t}\Lip_{\alpha,\beta_{t}}}{\alpha\underline{\theta}}\left(\eps_{t}+D\bar{\theta}\eps_{t}\right) \nonumber\\
&=\frac{(1+\bar{\theta}D)(1+\gamma_{t}\Lip_{\alpha,\beta_{t}})}{\alpha\underline{\theta}}\eps_{t}\eqdef\ce_{t}\label{eq:finalouter}.
\end{align}
Having obtained this new anchor point, we update the parameters $\eps_{t+1},\beta_{t+1},\gamma_{t+1}$ and restart the inner loop with these parameters and the new anchor point $w^{t+1}$. \\
We summarize our approach with the following method. 

\begin{algorithm}
\caption{Procedure $\texttt{IKM}(\alpha,\beta,\eps,\gamma,w)$}
(1) Let $t=1$, given $T,(\beta_t)_{t\in\N},(\eps_t)_{t\in\N},\K_{\alpha,\beta_t}(\eps,w^t),(\gamma_{k})_{k\in\N},(\tau_{k})_{k\in\N},(\theta_{k})_{k\in\N}$,\\
 $v^{0}=v^{1}=w^1$\\
(2) While $t<T$\\
(3) \For{$k=1,\ldots,\K_{\alpha,\beta_t}(\eps_t,w^t)$}{ 
 $z^{k}=v^{k}+\tau_{k}(v^{k}-v^{k-1})$\\ 
 $v^{k+1}=(1-\theta_{k})z^{k}+\theta_{k}\tilde{T}_{k}(z^{k})$
}
Return $v^{\K_{\alpha,\beta_t}(\eps_t,w^t)+1}$\\
(4) Update $w^{t+1}=v^{\K_{\alpha,\beta_{t}}(\eps_{t},w^{t})+1}$\\
(5) Set $t=t+1$ and go to (2)
\end{algorithm}

\section{Convergence Analysis}

\subsection{Inner Loop}
\noindent In the following, we simplify the notation by setting
\begin{align*}
&Q_k\eqdef 1-\theta_k+2\theta_kq_k^2 &E(\delta_k)\eqdef 2\theta_k\norm{\delta_k}^2 &\quad\lambda_k=\frac{1-\theta_k}{\theta_k}.
\end{align*}

\begin{assumption}\label{ass:parameterineq}
The parameter sequences employed in the inner loop satisfy
\begin{equation}\label{eq:parameterineq}
Q_k\tau_k(1+\tau_k)+\lambda_k\tau_k(1-\tau_k)-Q_k\lambda_{k-1}(1-\tau_{k-1})\leq 0.
\end{equation}
\end{assumption}

\begin{theorem}\label{th:MainInexact}
Assume that $(\theta_{k}\delta_{k})_{k}\in\ell^{2}_{+}(\N)$, and Assumptions \ref{ass:parameters} and \ref{ass:parameterineq} hold. If $(z^{k},v^{k})_{k}$ is generated by procedure $\texttt{IKM}(\alpha,\beta,\eps,w)$ without stopping criterion, then $(z^{k},v^{k})$ converge strongly to $\bar{u}_{\alpha,\beta}(w)$. Moreover $\sum_{k\geq 1}\norm{v^{k}-\bar{u}_{\alpha,\beta}(w)}^{2}<\infty$. 
\end{theorem}

\begin{proof}
To simplify the notation we abbreviate $\bar{u}\equiv\bar{u}_{\alpha,\beta}(w)$. We start with applying Lemma \ref{lem1} with $\xi=-\theta_k$

\begin{align*}
&\norm{v^{k+1}-\bar{u}}^2=\norm{(1-\theta_k)z^k+\theta_{k}\tilde{T}_{k}(z^k)-\bar{u}}^2\\
&=\norm{z^k-\bar{u}-\theta_k(-\tilde{T}_{k}(z^k)+z^k)}^2\\
&=(1-\theta_k)\norm{z^k-\bar{u}}^2-\theta_k(1-\theta_k)\norm{z^k-\tilde{T}_{k}(z^k)}^2\\
&\;\;\;\;\;\;\;\;+\theta_k\norm{\tilde{T}_{k}(z^k)-\bar{u}}^{2}.
\end{align*}

\noindent Next, we employ Lemma \ref{lem1} with $\xi=\tau_{k}$ to obtain
\begin{align*}
\norm{z^{k}-\bar{u}}^{2}&=\norm{v^{k}+\tau_{k}(v^{k}-v^{k-1})-\bar{u}}^{2}\\
&=(1+\tau_{k})\norm{v^{k}-\bar{u}}^{2}+\tau_{k}(1+\tau_{k})\norm{v^{k}-v^{k-1}}^{2}\\
&\;\;\;\;\;\;\;\;-\tau_{k}\norm{v^{k-1}-\bar{u}}^{2}.
\end{align*}

\noindent Using that $\bar{u}=T_{\gamma_k}(\bar{u})$ for all $k$ and the $q_k$-contraction property of the mapping $T_{\gamma_k}$, we obtain 
\begin{align*}
&\norm{\tilde{T}_{k}(z^k)-\bar{u}}^2=\norm{T_{\gamma_k}(z^k)+\delta_k-\bar{u}}^2\\
&\leq 2q_k^2\Big[(1+\tau_k)\norm{v^k-\bar{u}}^2+\tau_k(1+\tau_k)\norm{v^k-v^{k-1}}^2\\
&\;\;\;\;\;\;\;\;-\tau_k\norm{v^{k-1}-\bar{u}}^2\Big]+2\theta_{k}\norm{\delta_k}^2.
\end{align*}

\noindent Substituting all these terms, we can continue our development with 
\begin{align*}
\norm{v^{k+1}-\bar{u}}^{2}&\leq (1-\theta_{k}+2\theta_kq_k^{2})\Big[(1+\tau_{k})\norm{v^{k}-\bar{u}}^{2}\\
&\;\;\;+\tau_{k}(1+\tau_{k})\norm{v^{k}-v^{k-1}}^{2}-\tau_{k}\norm{v^{k-1}-\bar{u}}^{2} \Big]\\
&\;\;\;-\theta_{k}(1-\theta_{k})\norm{z^{k}-\tilde{T}_{k}(z^{k})}^{2}+2\theta_{k}\norm{\delta_k}^2.
\end{align*}

\noindent As a next step we take a look at the second term. Notice that $\tilde{T}_{k}(z^k)-z^k=\frac{v^{k+1}-z^k}{\theta_k}$. Applying again Lemma \ref{lem1} we get
\begin{align*}
-\theta_k^2 &\norm{\tilde{T}_{k}(z^k)-z^k}^2=-\norm{v^{k+1}-z^k}^2\\
&=-\norm{v^{k+1}-v^k-\tau_k(v^k-v^{k-1})}^2\\
&\leq-(1-\tau_k)\norm{v^{k+1}-v^k}^2+\tau_k(1-\tau_k)\norm{v^k-v^{k-1}}^2.
\end{align*}
Define $\lambda_{k}\eqdef (1-\theta_{k})/\theta_{k}$, with $\theta_k\in(0,1)$ for all $k$ and multiply both sides in the display above by $\lambda_{k}$ in order to arrive at 
\begin{align*}
-\theta_k(1-\theta_k)\norm{\tilde{T}_{k}(z^k)-z^k}^2\leq&-\lambda_k(1-\tau_k)\norm{v^{k+1}-v^k}^2\\
&+\lambda_k\tau_k(1-\tau_k)\norm{v^k-v^{k-1}}^2.
\end{align*}

\noindent Combining all these estimates, together with some elementary algebra, we can continue our energy bound as 
\begin{align*}
\norm{v^{k+1}-\bar{u}}^2&\leq Q_k\Big[(1+\tau_{k})\norm{v^{k}-\bar{u}}^{2}+\tau_{k}(1+\tau_{k})\norm{v^{k}-v^{k-1}}^{2}\\
&\; -\tau_{k}\norm{v^{k-1}-\bar{u}}^{2} \Big]\ - \lambda_k(1-\tau_k)\norm{v^{k+1}-v^k}^2\\
&+\lambda_k\tau_k(1-\tau_k)\norm{v^k-v^{k-1}}^2+2\theta_k\norm{\delta_k}^2.
\end{align*}


\noindent Rearranging the terms and evoking Assumption \eqref{ass:parameterineq} we get
\begin{align*}
&\norm{v^{k+1}-\bar{u}}^2+\lambda_k(1-\tau_k)\norm{v^{k+1}-v^k}^2\\
&\leq Q_k(1+\tau_k)\norm{v^{k}-\bar{u}}^2-Q_k\tau_k\norm{v^{k-1}-\bar{u}}^2\\
&+Q_k\lambda_{k-1}(1-\tau_{k-1})\norm{v^{k}-v^{k-1}}^2+E(\delta_k).
\end{align*}

\noindent Subtracting $\tau_k\norm{v^k-\bar{u}}^2$ from both sides and using that $(\tau_{k})_{k}$ is monotonically increasing, we can continue our estimation by 
\begin{align*}
\norm{v^{k+1}-\bar{u}}^2&-\tau_k\norm{v^k-\bar{u}}^2+\lambda_k(1-\tau_k)\norm{v^{k+1}-v^k}^2\\
&\leq Q_k\norm{v^k-\bar{u}}^2-Q_k\tau_{k-1}\norm{v^{k-1}-\bar{u}}^2\\
&\;\;\;\;\;+Q_k\lambda_{k-1}(1-\tau_{k-1})\norm{v^{k}-v^{k-1}}^2+E(\delta_k).
\end{align*}

\noindent Consequently, if we define 
\begin{equation*}
V_{k}\eqdef\norm{v^{k}-\bar{u}}^{2}-\tau_{k-1}\norm{v^{k-1}-\bar{u}}^{2}+\lambda_{k-1}(1-\tau_{k-1}) \norm{v^{k}-v^{k-1}}^{2}, 
\end{equation*}
we obtain $V_{k+1}\leq Q_{k} V_{k}+E(\delta_k)$ for all $k\geq 1$. Notice that $E(\delta_k)$ is indeed summable and $Q_k\in(0,1)$ and therefore we can  apply Lemma \ref{lem3} and we obtain $\sum_{k=1}^{\infty}[V_k]_+<\infty$, 
and hence $[V_k]_+$ converges to $0$. Moreover, 
\[
V_k\geq\norm{v^k-\bar{u}}^2-\tau_{k-1}\norm{v^{k-1}-\bar{u}}^{2}\geq \norm{v^k-\bar{u}}^2-\bar{\tau}\norm{v^{k-1}-\bar{u}}^{2}, 
\]
where $\bar{\tau}=\sup_{k}\tau_{k}\in(0,1)$. Applying again Lemma \ref{lem3} with the identification $w_k=\norm{v^k-\bar{u}}^2$ and $d=\bar{\tau}$, we can conclude $\sum_{k=1}^{\infty}\norm{v^k-\bar{u}}^2<\infty$ and therefore $\norm{v^k-\bar{u}}^2\to 0$ as $k\to\infty$. 
\end{proof}

\subsection{Outer Loop}

\noindent We now proof the asymptotic convergence of the sequence $(w^{t})_{t\in\N}$ produced in the outer loop of the method. Let $\scrS_{1}\eqdef \zer(\opG+\NC_{\scrS_{0}})$. Note that $\scrS_{1}\subseteq\scrS_{0}$, and $\scrS_{0}$ is a closed convex and nonempty set. 

\begin{theorem}
Let Assumptions \ref{ass:1} and \ref{ass:slowcontrol} hold. Then, the outer loop of our method $\texttt{IKM}(\alpha,\beta,\eps,\gamma,w)$ produces a bounded sequence $\{w^{t}\}_{t\in\N}$ whose weak limit points are contained in $\scrS_{1}$. 
\end{theorem}
\begin{proof}
Define the anchor function $h_{t}=\frac{1}{2}\norm{w^{t}-\Pi_{\scrS_{1}}(w^{t})}^{2}.$ We then have 
\begin{align*}
h_{t+1}-h_{t}&=\frac{1}{2}\norm{w^{t+1}-\Pi_{\scrS_{1}}(w^{t+1})}^{2}-\frac{1}{2}\norm{w^{t}-\Pi_{\scrS_{1}}(w^{t})}^{2}\\
&\leq \frac{1}{2}\norm{w^{t+1}-\Pi_{\scrS_{1}}(w^{t})}^{2}-\frac{1}{2}\norm{w^{t}-\Pi_{\scrS_{1}}(w^{t})}^{2} \\
&=-\frac{1}{2}\norm{w^{t+1}-w^{t}}^{2}+\inner{w^{t+1}-w^{t},w^{t+1}-\Pi_{\scrS_{1}}(w^{t})}.
\end{align*}
For all $t\in\N$ we define the unique solution 
\begin{equation}
\bar{u}^{t+1}=\zer(A+\Phi_{t}(\cdot,w^{t})), 
\end{equation}
where $\Phi_{t}(\cdot,w^{t})\equiv\Phi_{\alpha,\beta_{t}}(\cdot,w^{t})=\opF(\cdot)+\beta_{t}\opG(\cdot)+\alpha(\cdot-w^{t})$. Let $(v,a)\in\gr(\opA)$. By monotonicity of $\opA$ and the fact that $(\bar{u}^{t+1},-\Phi_{t}(\bar{u}^{t+1},w^{t}))\in\gr(\opA)$, we conclude that 
\begin{align*}
\inner{\opF(\bar{u}^{t+1})+a+\alpha(\bar{u}^{t+1}-w^{t}),v-\bar{u}^{t+1}}\geq\beta_{t}\inner{\opG(\bar{u}^{t+1}),\bar{u}^{t+1}-v}.
\end{align*}
Therefore, for all $(v,a)\in\gr(\opA)$, 
\begin{align*}
&\alpha\inner{w^{t}-w^{t+1},v-w^{t+1}}\\
&\leq \beta_{t}\inner{\opG(\bar{u}^{t+1}),v-\bar{u}^{t+1}}+\inner{\opF(\bar{u}^{t+1})+a,v-\bar{u}^{t+1}}\\
&\;\;\;+\alpha\left(\inner{\bar{u}^{t+1}-w^{t+1},v-w^{t+1}}+\inner{w^{t}-\bar{u}^{t+1},\bar{u}^{t+1}-w^{t+1}}\right) .
\end{align*}
We now add and subtract terms to estimate the inner products as follows:
\begin{align*}
&\inner{\opG(\bar{u}^{t+1}),v-\bar{u}^{t+1}}\\
&\leq \inner{\opG(w^{t+1}),v-w^{t+1}}+\ce_{t}\left(\norm{\opG(\bar{u}^{t+1})}+\Lip_{\opG}\norm{v-w^{t+1}}\right). 
\end{align*}
In the same way, 
\begin{align*}
&\inner{\opF(\bar{u}^{t+1}),v-\bar{u}^{t+1}}\\
&\leq \inner{\opF(w^{t+1}),v-w^{t+1}}+\ce_{t}\left(\Lip_{\opF}\norm{v-w^{t+1}}+\norm{\opF(\bar{u}^{t+1})}\right),
\end{align*}
and $\inner{a,v-\bar{u}^{t+1}}\leq \inner{a,v-w^{t+1}}+\norm{a}\ce_{t}.$ Moreover we have 
\begin{align*}
&\alpha\left(\inner{\bar{u}^{t+1}-w^{t+1},v-w^{t+1}}+\inner{w^{t}-\bar{u}^{t+1},\bar{u}^{t+1}-w^{t+1}}\right)\\
&\leq\alpha \ce_t\left(\norm{v-w^{t+1}}+\norm{w^{t}-\bar{u}^{t+1}}\right).
\end{align*}
Combining all these estimates, we can continue from the above to obtain 
\begin{equation}\label{eq:mainestimate}
\begin{split}
&\alpha\inner{w^{t}-w^{t+1},v-w^{t+1}}\leq\beta_{t}\inner{\opG(w^{t+1}),v-w^{t+1}}\\
&\;\;+\inner{\opF(w^{t+1})+a,v-w^{t+1}}+\ce_{t}R_{t}(v,a). 
\end{split}
\end{equation}
where $R_{t}:\gr(\opA)\to\R$ is defined $\forall (v,a)\in\gr(\opA)$ as  
\begin{align}
R_{t}(v,a)&\eqdef\norm{a}+\Lip_{\alpha ,\beta}\norm{v-w^{t+1}}+\norm{\opF(\bar{u}^{t+1})}\nonumber\\
&\;\;\;+\beta_t\norm{\opG(\bar{u}^{t+1})}+\alpha\norm{\bar{u}^{t+1}-w^t}. 
\end{align}
Using the inclusion $\scrS_{1}\subseteq\scrS_{0}$, then we can particularise the above estimate by using the point $(v,a)=(\Pi_{\scrS_{1}}(w^{t}),-\opF(\Pi_{\scrS_{1}}(w^{t})))\in\gr(\opA)$,\footnote{Indeed, $\tilde{w}^{t}=\Pi_{\scrS_{1}}(w^{t})\in\scrS_{0}$ and therefore $-\opF(\tilde{w}^{t})\in\opA(\tilde{w}^{t})$.} to obtain 
\begin{align}
h_{t+1} &-h_{t}\leq -\frac{1}{2}\norm{w^{t+1}-w^{t}}^{2}+\frac{\beta_{t}}{\alpha}\inner{\opG(w^{t+1}),\Pi_{\scrS_{1}}(w^{t})-w^{t+1}} \nonumber\\
&+\ce_{t}R_{t}/\alpha+\frac{1}{\alpha}\inner{\opF(w^{t+1})-\opF(\Pi_{\scrS_{1}}(w^{t})),\Pi_{\scrS_{1}}(w^{t})-w^{t+1}} \nonumber\\
&\leq -\frac{1}{2}\norm{w^{t+1}-w^{t}}^{2}+\frac{\beta_{t}}{\alpha}\Gamma_{t+1}+\ce_{t}R_{t}/\alpha.\label{eq:mainbound}
\end{align}
where 
\begin{equation}\label{eq:Gap}
\Gamma_{t}:=\inner{\opG(w^{t}),\Pi_{\scrS_{1}}(w^{t-1})-w^{t}}. 
\end{equation} 
%
and $R_{t}\equiv R_{t}(\Pi_{\scrS_{1}}(w^{t}),-\opF(\Pi_{\scrS_{1}}(w^{t})))$. 
To complete the proof, we consider several cases: 
\paragraph*{Case 1} Assume there exists $T_{0}\in\N$ such that for all $t\geq T_{0}$ it holds $\beta_{t}\Gamma_{t+1}+R_{t}\ce_{t}\leq 0.$ In other words, the set 
\[
I_{0}:=\{t\in\N\vert \beta_{t}\Gamma_{t+1}+R_{t}\ce_{t}\leq 0\}
\]
contains $T_{0}$, and all subsequent iterates. Then, for all $t\in I_{0}$ \eqref{eq:mainbound} immediately yields  
\[
h_{t+1}-h_{t}\leq -\frac{1}{2}\norm{w^{t+1}-w^{t}}^{2}\leq 0.
\]
Hence, $(h_{t})_{t\in I_{0}}$ is monotonically decreasing, so that $\lim_{t\to\infty,t\in I_{0}}h_{t}$ exists. Furthermore, 
\[
\frac{1}{2}\sum_{t=T_{0}}^{T-1}\norm{w^{t+1}-w^{t}}^{2}\leq h_{T_{0}}-h_{T}\leq h_{T_{0}}, 
\]
which implies $\norm{w^{t+1}-w^{t}}\to 0$ for  $t\to\infty,t\in I_{0}.$ Since $\beta_{t}\Gamma_{t+1}\leq \beta_{t}\Gamma_{t+1}+R_{t}\ce_{t}\leq 0 $ for all $t\in I_{0}$, we deduce that $\limsup_{t\to\infty,t\in I_{0}}\Gamma_{t}\leq 0$. We claim that 
\begin{equation}\label{eq:Gamma0}
\limsup_{t\to\infty,t\in I_{0}}\Gamma_{t}=0.
\end{equation}
Suppose not, i.e assume there exists $\bar{\Gamma}<0$ such that $\Gamma_{t}\leq\bar{\Gamma}<0$ for all $t$ sufficiently large. Since $\bar{u}^{t+1}$ and $w^{t}$ are both contained in $\dom(\opA)$ for all $t\geq 1$, and the operators $\opF$ and $\opG$ are Lipschitz, it follows that $(R_{t})_{t\in I_{0}}$ is a bounded sequence. Hence, there exists $M\in\N$ for which $R_{t}\leq M$. This implies that for $T>T_{1}\geq T_{0}$ sufficiently large, we have 
\[
h_{T}-h_{T_{1}}\leq \frac{\bar{\Gamma}}{\alpha}\sum_{t=T_{1}}^{T-1}\beta_{t}+\frac{M}{\alpha}\sum_{t=T_{1}}^{T-1}\ce_{t}.
\]
Assumption \ref{ass:slowcontrol} guarantees that we can choose $T_{1}$ large enough so that $\ce_{t}\leq \frac{-\beta_{t}}{2M}\bar{\Gamma}$. Hence, we can continue from the above display to obtain the bound 
\[
h_{T}-h_{T_{1}}\leq \frac{\bar{\Gamma}}{2\alpha}\sum_{t=T_{1}}^{T-1}\beta_{t}.
\]
Since $(\beta_{t})_{t\in\N}$ is not summable, we conclude $h_{T}\to-\infty$, a contradiction. Hence, \eqref{eq:Gamma0} holds true. 

\noindent Now, let $(w^{n})_{n\in\mathcal{N}}$ be a weakly converging subsequence of $(w^{t})_{t\in\N}$ with weak limit $\tilde{w}$. We can extract such a weakly converging subsequence since, by Assumption \ref{ass:1},  $\dom(\opA)$  is closed and bounded. Suppose $\tilde{w}\notin\zer(\opA+\opF)$. By maximal monotonicity of $\opA+\opF$, there exists $(v,a)\in\gr(\opA+\opF)$ such that 
\begin{equation}\label{eq:contra}
\inner{a-0,v-\tilde{w}}=\inner{a,v-\tilde{w}}<0.
\end{equation}

\noindent For this pair $(v,a)\in\gr(\opA+\opF)$, inequality \eqref{eq:mainestimate} implies 
\begin{align*}
&\alpha\inner{w^{n}-w^{n+1},v-w^{n+1}}\\&\leq\beta_{n}\inner{\opG(w^{n+1}),v-w^{n+1}}+\inner{a,v-w^{n+1}}+\ce_{n}R_{n}(v,a). 
\end{align*}

\noindent Since $\lim_{t\to\infty,t\in I_{0}}\norm{w^{t+1}-w^{t}}=0$, and the sequence $\left(\norm{v-w^{t}}\right)_{t\in I_{0}}$ is bounded, we have 
\[
\inner{w^{n}-w^{n+1},v-w^{n+1}}\leq\norm{w^{n}-w^{n+1}}\cdot\norm{v-w^{n+1}}\to 0. 
\]
Since $\opG$ is monotone, we further see 
\[
\beta_{t'}\inner{\opG(w^{n+1}),v-w^{n+1}}\leq \beta_{n}\inner{\opG(v),v-w^{n+1}}\to 0.
\]

\noindent Moreover, since by Assumption \ref{ass:slowcontrol} $\ce_t\to 0$, we conclude $\ce_{n}R_{n}\to 0.$
Thus, passing to the limit along the weakly converging subsequence, we see that we get
\[
\inner{a,v-\tilde{w}}\geq 0.
\]
This is a contradiction to \eqref{eq:contra}. We therefore conclude that $\tilde{w}\in \scrS_{0}$.\\
\noindent Now consider a subsequence of $(w^{n})_{n\in\mathcal{N}}$, $(w^{m})_{m\in\mathcal{N}'}$ such that 
\[
\Pi_{\scrS_{1}}(w^{m})\wlim \hat{w} \quad\text{and}\quad w^{m+1}\wlim\bar{w}.
\]
Since, $w^{m}\wlim \tilde{w}$ and $w^{m}-w^{m +1}\rightarrow 0$ we have $\tilde{w}=\bar{w}$ by the uniqueness of the weak limit. Moreover the following holds
\[
\Pi_{\scrS_{1}}(w^{m})-w^{m}\wlim\hat{w}-\tilde{w}.
\]

By the definition of $h_{m}$, which we know to converges, we deduce that $\Pi_{\scrS_{1}}(w^{m})-w^{m}$ converges weakly and in norm Therefore it converges strongly. \\
\noindent Now by the weak sequentially continuity of $\opG$, the continuity of the inner product  and  \eqref{eq:Gamma0} the following holds
\[
0=\lim_{m\rightarrow\infty}\inner{\opG(w^{m+1}),\Pi_{\scrS_{1}}(w^{m})-w^{m+1}}=\inner{\opG{\tilde{w}},\hat{w}-\tilde{w}}.
\]
Furthermore, $\scrS_{0}$ is convex and (weakly) closed as it is the zero set of a maximally monotone operator, hence the normal cone of $\scrS_{0}$ is maximally monotone. On the other hand, $\opG$ is monotone and Lipschitz which implies that $\opG+\NC_{\scrS_0}$ is maximally monotone. Therefore $\scrS_{1}=\zer(\opG+\NC_{\scrS_0})$ is convex and weakly closed. It follow that $\tilde{w}\in\scrS_{1}$. \\

\noindent But then, via Lemma \ref{lem:solution}, we see $\tilde{w}\in\scrS_{1}$.

\paragraph*{Case 2} Consider the sets 
\begin{align*}
&I_{1}:=\{t\in\N\vert\;\beta_{t}\Gamma_{t+1}+R_{t}\ce_{t}>0\}=\N\setminus I_{0},\text{ and }\\ 
&I_{2}:=\{t\in I_{1}\vert -\frac{1}{2}\norm{w^{t+1}-w^{t}}^{2}+\frac{\beta_{t}}{\alpha}\Gamma_{t+1}+\frac{R_{t}\ce_{t}}{\alpha}>0\}\subseteq I_{1}.
\end{align*}
Assume that both sets are infinite. It then suffices to show that $(h_{t})_{t\in I_{2}}$ converges to $0$. To this end, observe first that for every $t\in I_{2}$, it holds that 
\begin{equation}\label{eq:estimateCase2}
\beta_{t}\Gamma_{t+1}+R_{t}\ce_{t}>\frac{\alpha}{2}\norm{w^{t+1}-w^{t}}^{2} .
\end{equation}
The sequence of temporal solution candidates $(w^{t})_{t\in I_{2}}$ admits a converging subsequence. Hence, we can take a weakly converging subsequence $(w^{\tau})_{\tau\in\bT}$ with $\bT\subseteq I_{2}$ and weak limit $\tilde{w}$. Then, $(\Gamma_{\tau})_{\tau\in\bT}$ is bounded and thus the last display shows that
$\lim_{\tau\to\infty,\tau\in\bT}\norm{w^{\tau+1}-w^{\tau}}=0$. Reasoning exactly as in Case 1, we conclude $\tilde{w}\in\scrS_{0}$. Additionally, from \eqref{eq:estimateCase2}, we see
\[
\Gamma_{\tau+1}+\frac{R_{\tau}\ce_{\tau}}{\beta_{\tau}}\geq \frac{1}{2\beta_{\tau}}\norm{w^{\tau+1}-w^{\tau}}^{2}\geq 0.
\]
The first addendum on the left-hand side is bounded, and by the slow control assumption \ref{ass:slowcontrol} the second addendum is a null sequence. Hence, $\left\{\frac{\norm{w^{\tau+1}-w^{\tau}}^{2}}{2\beta_{\tau}}\right\}_{\tau\in\bT}$ is bounded. In particular, we have by weak sequential continuity of $\opG$:
\[
\inner{\opG(\tilde{w}),\Pi_{\scrS_{1}}(\tilde{w})-\tilde{w}}\geq 0,
\]
which implies $\tilde{w}\in\scrS_{1}$ by citing Lemma \ref{lem:solution}.

\paragraph*{Case 3} $I_{1}$ is infinite while $I_{2}$ is finite. In this case, $I_{1}\setminus I_{2}$ is infinite and hence the only relevant phase for the asymptotic analysis of the algorithm. The sequence $(h_{t})_{t\in I_{1}\setminus I_{2}}$ is decreasing and hence converging. This means $h_{t+1}-h_{t}\to 0$ for $t\to\infty,t\in I_{1}\setminus I_{2}\equiv I_{3}$. 
Using \eqref{eq:mainbound}, we see 
\begin{align*}
&\underbrace{h_{t+1}-h_{t}}_{\to 0,\text{as }t\to \infty,t\in I_{3}} \\
&\leq \underbrace{-\frac{1}{2}\norm{w^{t+1}-w^{t}}^{2}+\frac{\beta_{t}}{\alpha}\inner{\opG(w^{t+1}),\Pi_{\scrS_{1}}(w^{t})-w^{t+1}}+\frac{\ce_{t}R_{t}}{\alpha}}_{\leq 0\text{ for }t\in I_{3}}.
\end{align*}
Since the sequence $(w^{t})_{t\in I_{3}}$ is bounded and $\beta_{t}\to0$ as well as $\ce_{t}\to0$ as $t\to\infty$, it follows 
\[
\lim_{t\to\infty,t\in I_{3}}\norm{w^{t+1}-w^{t}}=0.
\]
Let now $\tilde{w}$ denote a weak limit point of $(w^{t})_{t\in I_{3}}$. As in case 2, we see $\tilde{w}\in\scrS_{1}$, hence $h_{t}\to 0$ for $t\to\infty,t\in I_{3}$. But since the entire sequence converges, we conclude $\lim_{t\to\infty}h_{t}=0$. 
 \end{proof}

\section{Numerical Experiments}

\noindent We consider the two-player zero-sum game introduced in \cite{samadi2025improved}, given by
\begin{align*}
    &\begin{dcases}
        &\min_{x_1}f_1(x_1, x_2)\coloneqq 20-0.1x_1x_2+x_1 \\
        &\text{s.t.}~ x_1\in X_1\coloneqq [11, 60]
    \end{dcases}\\
    &\begin{dcases}
        &\min_{x_2}f_2(x_1, x_2)\coloneqq -20+0.1x_1x_2-x_1 \\
        &\text{s.t.}~ x_2\in X_2\coloneqq [10, 50].
    \end{dcases}
\end{align*}
We seek a saddle point $(x_1^*, x_2^*)\in X\coloneqq X_1\times X_2$ of the function $f(x_1, x_2)\coloneqq 20-0.1 x_1x_2+x_1$, namely a point that satisfies
\[
    f(x_1^*, x_2)\le f(x_1^*, x_2^*)\le f(x_1, x_2^*)\quad \text{for all $(x_1, x_2)\in X$.}
\]
The solution set is characterized through the inclusion 
\[
    0\in F+\NC_X =\begin{bmatrix}
        0 & -0.1 \\ 0.1 & 0
    \end{bmatrix}\begin{bmatrix}
        x_1 \\ x_2
    \end{bmatrix} + \begin{bmatrix}
        1 \\ 0
    \end{bmatrix} + \NC_X.
\]
The set of solutions is  $X_1\times \{10\}$. We solve the problem 
\[
    \min \left\{\phi(x_1, x_2)\coloneq (x_1, x_2)\in \zer(F+\NC_X)\right\},
\]
where $\phi(x)=\tfrac12\|x\|^2$, whose analytical solution is $(11, 10)$.
Equivalently, we aim to find a point $(x_1^*, x_2^*)\in \zer(F+NC_X)$ such that $\langle \nabla \phi (x_1^*, x_2^*), (w_1, w_2) - (x_1^*, x_2^*)\rangle \ge 0\; \text{for all $(w_1, w_2)\in \zer(F+\NC_X)$}$.
The subproblem of Algorithm $\texttt{IKM}(\alpha,\beta,\eps,\gamma,w)$ is thus given by determining a solution to 
\[  
    \overline u^{t+1}=\zer(F+\NC_X+\beta_t \nabla \phi + \alpha (I-w^t)).
\]
We define the smooth part to be $\Phi=F+\beta_t \nabla \phi + \alpha (I-w^t)$ which is Lipschitz continuous with constant $L = 0.1 + \beta_t + \alpha$. We note that $\Phi$ is $\alpha$-strongly monotone and select a step-size $\gamma={\alpha}/{L^2}$, such that the operator $T$ defined by the forward-backward map is contractive with constant $\sqrt{1-{\alpha^2}/{L^2}}$. We consider the relaxation parameter of the inner loop $\theta_k\equiv \theta\in (0,1)$ to be constant. Across inner loops, we consider the acceleration parameter $\tau_k\equiv \tau$ to be constant and to be the largest value satisfying equation \eqref{eq:parameterineq}. We assume $\beta_t=(t+1)^{-\eta}$, where $\eta=0.55$, and a stopping criterion of  $\varepsilon_t=\bar\varepsilon\cdot {(t+1)^{-2}}$, where $\bar\varepsilon=10^{-3}$. We run a total of $100$ iterations. The iterates for various starting points and different values for $\alpha$, along with the feasible region $X$ are shown in figures \ref{fig:1}, \ref{fig:2} and \ref{fig:3}.

\begin{figure}
\includegraphics[width=8.5cm, height=6cm]{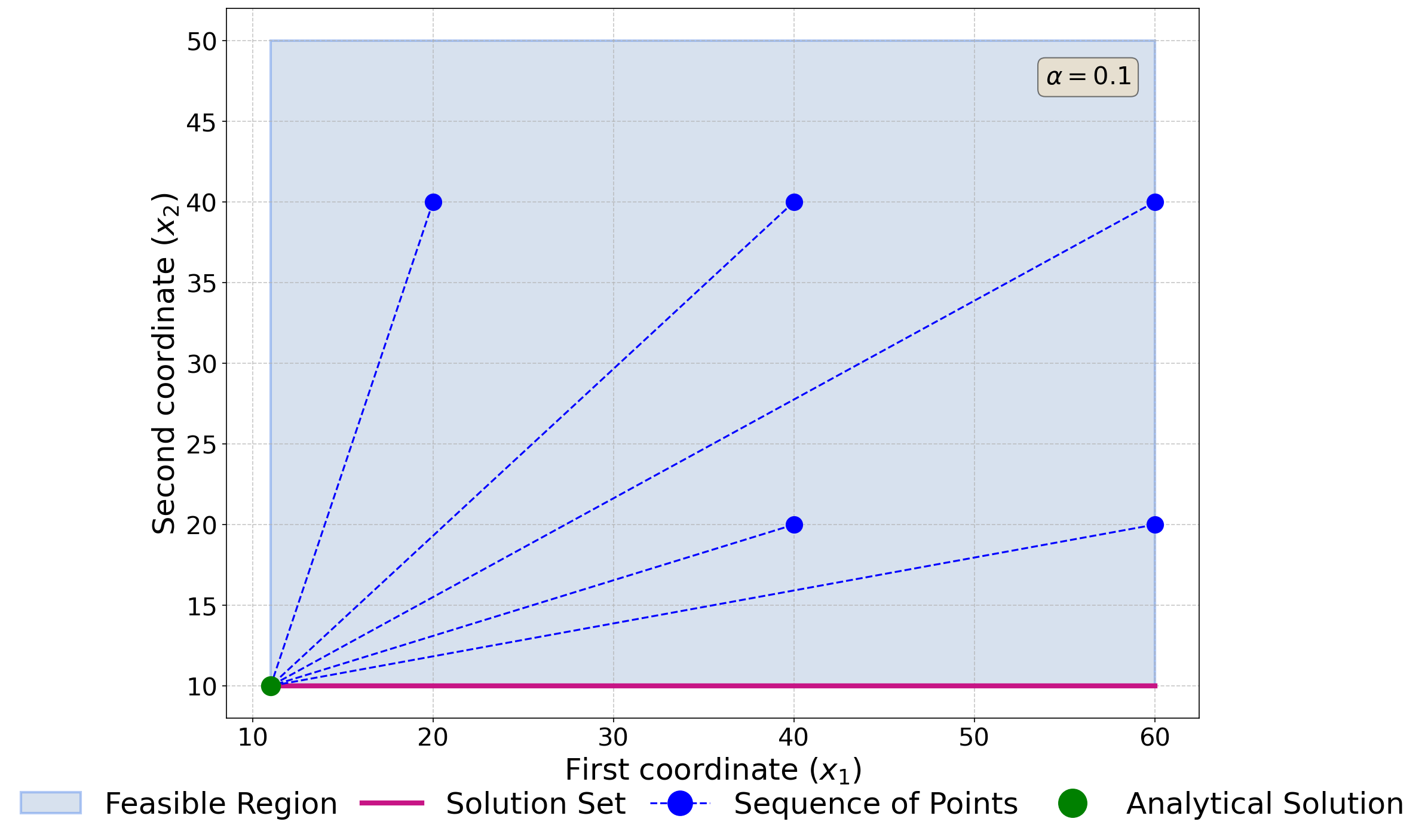}
\caption{Iterates for various initial points and $\alpha=0.1$.}
\label{fig:1}
\end{figure}

\begin{figure}
\includegraphics[width=8.5cm, height=6cm]{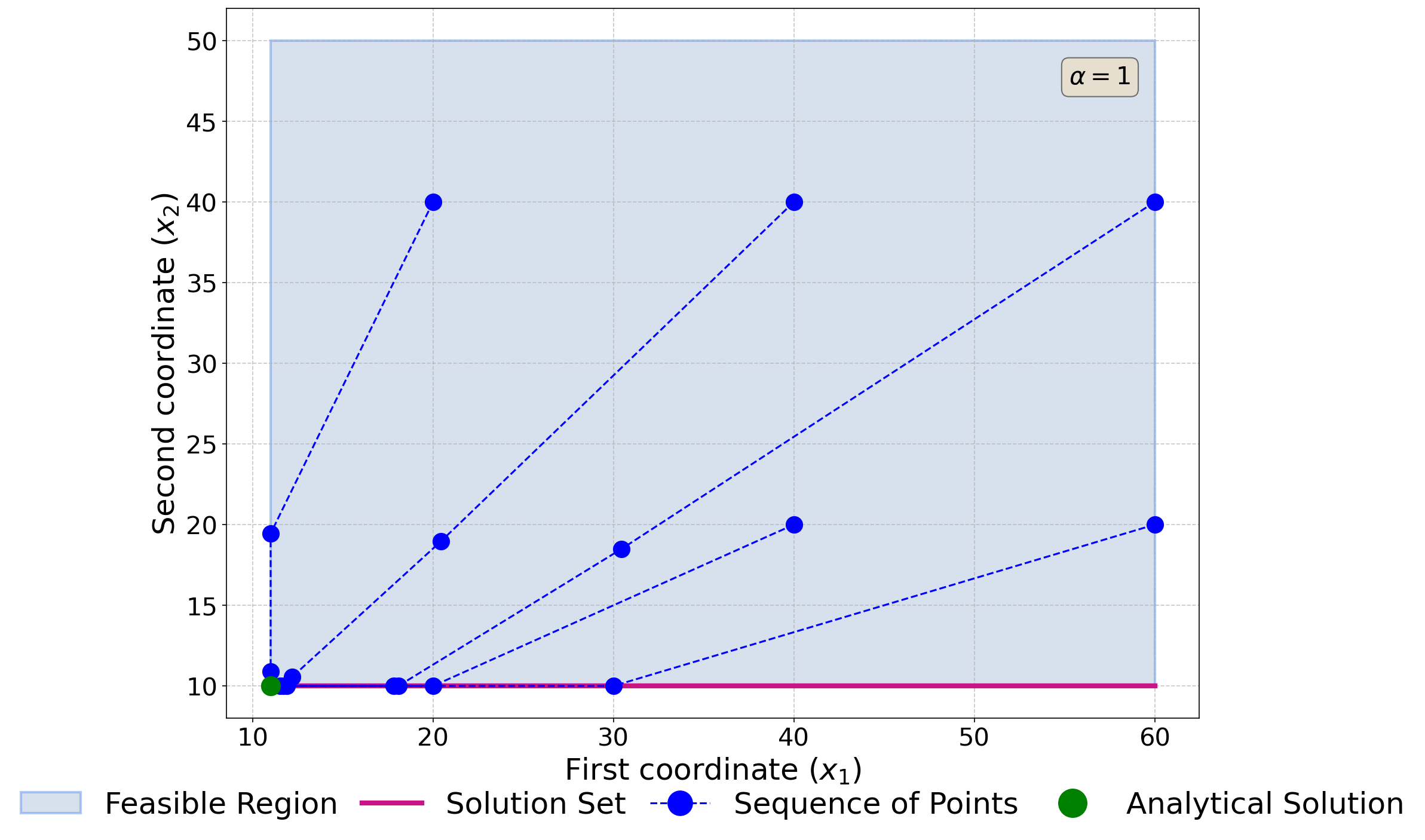}
\caption{Iterates for various initial points and $\alpha=1$.}
\label{fig:2}
\end{figure}

\begin{figure}
\includegraphics[width=8.5cm, height=6cm]{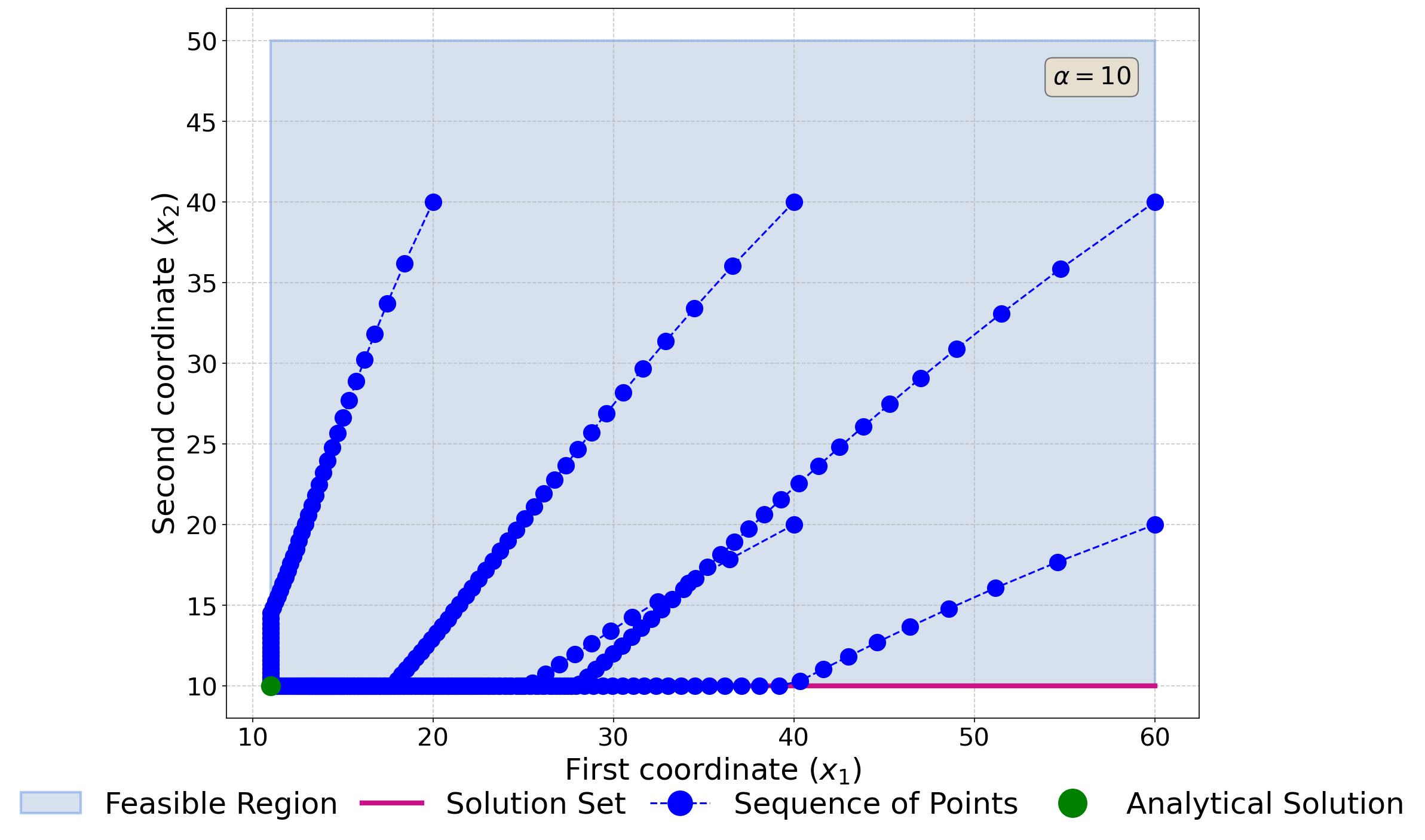}
\caption{Iterates for various initial points and $\alpha=10$.}
\label{fig:3}
\end{figure}


\begin{thebibliography}{10}
\providecommand{\url}[1]{#1}
\csname url@rmstyle\endcsname
\providecommand{\newblock}{\relax}
\providecommand{\bibinfo}[2]{#2}
\providecommand\BIBentrySTDinterwordspacing{\spaceskip=0pt\relax}
\providecommand\BIBentryALTinterwordstretchfactor{4}
\providecommand\BIBentryALTinterwordspacing{\spaceskip=\fontdimen2\font plus
\BIBentryALTinterwordstretchfactor\fontdimen3\font minus
  \fontdimen4\font\relax}
\providecommand\BIBforeignlanguage[2]{{%
\expandafter\ifx\csname l@#1\endcsname\relax
\typeout{** WARNING: IEEEtran.bst: No hyphenation pattern has been}%
\typeout{** loaded for the language `#1'. Using the pattern for}%
\typeout{** the default language instead.}%
\else
\language=\csname l@#1\endcsname
\fi
#2}}

\bibitem{Facchinei:2014aa}
\BIBentryALTinterwordspacing
F.~Facchinei, J.-S. Pang, G.~Scutari, and L.~Lampariello, ``Vi-constrained
  hemivariational inequalities: distributed algorithms and power control in
  ad-hoc networks,'' \emph{Mathematical Programming}, vol. 145, no.~1, pp.
  59--96, 2014. [Online]. Available:
  \url{https://doi.org/10.1007/s10107-013-0640-5}
\BIBentrySTDinterwordspacing

\bibitem{Benenati:2022aa}
E.~Benenati, W.~Ananduta, and S.~Grammatico, ``On the optimal selection of
  generalized nash equilibria in linearly coupled aggregative games,''
  \emph{IEEE 61st Conference on Decision and Control (CDC)}, pp. 6389--6394,
  2022.

\bibitem{Benenati:2023aa}
------, ``A semi-decentralized tikhonov-based algorithm for optimal generalized
  nash equilibrium selection,'' \emph{62nd IEEE Conference on Decision and
  Control (CDC)}, pp. 4243--4248, 2023.

\bibitem{Kaushik:2021aa}
H.~D. Kaushik and F.~Yousefian, ``A method with convergence rates for
  optimization problems with variational inequality constraints,'' \emph{SIAM
  Journal on Optimization}, vol.~31, no.~3, pp. 2171--2198, 2021.

\bibitem{Yamada2011MinimizingTM}
\BIBentryALTinterwordspacing
I.~Yamada, M.~Yukawa, and M.~Yamagishi, ``Minimizing the moreau envelope of
  nonsmooth convex functions over the fixed point set of certain
  quasi-nonexpansive mappings,'' in \emph{Fixed-Point Algorithms for Inverse
  Problems in Science and Engineering}, 2011. [Online]. Available:
  \url{https://api.semanticscholar.org/CorpusID:124791335}
\BIBentrySTDinterwordspacing

\bibitem{Solodov:2007aa}
M.~Solodov, ``An explicit descent method for bilevel convex optimization,''
  \emph{Journal of Convex Analysis}, vol.~14, no.~2, p. 227, 2007.

\bibitem{yousefian2021bilevel}
F.~Yousefian, ``Bilevel distributed optimization in directed networks,'' in
  \emph{2021 American Control Conference (ACC)}.\hskip 1em plus 0.5em minus
  0.4em\relax IEEE, 2021, pp. 2230--2235.

\bibitem{Pang:2008aa}
J.~S. Pang, G.~Scutari, F.~Facchinei, and C.~Wang, ``Distributed power
  allocation with rate constraints in gaussian parallel interference
  channels,'' \emph{IEEE Transactions on Information Theory}, vol.~54, no.~8,
  pp. 3471--3489, 2008.

\bibitem{dempe2010optimality}
S.~Dempe, N.~Dinh, and J.~Dutta, ``Optimality conditions for a simple convex
  bilevel programming problem,'' in \emph{Variational Analysis and Generalized
  Differentiation in Optimization and Control: In Honor of Boris S.
  Mordukhovich}.\hskip 1em plus 0.5em minus 0.4em\relax Springer, 2010, pp.
  149--161.

\bibitem{Dempe:2021aa}
\BIBentryALTinterwordspacing
S.~Dempe, N.~Dinh, J.~Dutta, and T.~Pandit, ``Simple bilevel programming and
  extensions,'' \emph{Mathematical Programming}, vol. 188, no.~1, pp. 227--253,
  2021. [Online]. Available: \url{https://doi.org/10.1007/s10107-020-01509-x}
\BIBentrySTDinterwordspacing

\bibitem{SabSht17}
\BIBentryALTinterwordspacing
S.~Sabach and S.~Shtern, ``A first order method for solving convex bilevel
  optimization problems,'' \emph{SIAM Journal on Optimization}, vol.~27, no.~2,
  pp. 640--660, 2017. [Online]. Available:
  \url{https://doi.org/10.1137/16M105592X}
\BIBentrySTDinterwordspacing

\bibitem{Alart:1991aa}
\BIBentryALTinterwordspacing
P.~Alart and B.~Lemaire, ``Penalization in non-classical convex programming via
  variational convergence,'' \emph{Mathematical Programming}, vol.~51, no.~1,
  pp. 307--331, 1991. [Online]. Available:
  \url{https://doi.org/10.1007/BF01586942}
\BIBentrySTDinterwordspacing

\bibitem{Bahraoui:1994aa}
\BIBentryALTinterwordspacing
M.~A. Bahraoui and B.~Lemaire, ``Convergence of diagonally stationary sequences
  in convex optimization,'' \emph{Set-Valued Analysis}, vol.~2, no.~1, pp.
  49--61, 1994. [Online]. Available: \url{https://doi.org/10.1007/BF01027092}
\BIBentrySTDinterwordspacing

\bibitem{Cominetti:1997aa}
\BIBentryALTinterwordspacing
R.~Cominetti, ``Coupling the proximal point algorithm with approximation
  methods,'' \emph{Journal of Optimization Theory and Applications}, vol.~95,
  no.~3, pp. 581--600, 1997. [Online]. Available:
  \url{https://doi.org/10.1023/A:1022621905645}
\BIBentrySTDinterwordspacing

\bibitem{Lorenz:2015aa}
\BIBentryALTinterwordspacing
D.~A. Lorenz and T.~Pock, ``An inertial forward-backward algorithm for monotone
  inclusions,'' \emph{Journal of Mathematical Imaging and Vision}, vol.~51,
  no.~2, pp. 311--325, 2015. [Online]. Available:
  \url{https://doi.org/10.1007/s10851-014-0523-2}
\BIBentrySTDinterwordspacing

\bibitem{Cortild:2024aa}
D.~Cortild and J.~Peypouquet, ``Krasnoselskii-mann iterations: Inertia,
  perturbations and approximation,'' \emph{arXiv preprint arXiv:2401.16870},
  2024.

\bibitem{villa2013accelerated}
S.~Villa, S.~Salzo, L.~Baldassarre, and A.~Verri, ``Accelerated and inexact
  forward-backward algorithms,'' \emph{SIAM Journal on Optimization}, vol.~23,
  no.~3, pp. 1607--1633, 2013.

\bibitem{Apidopoulos:2020aa}
\BIBentryALTinterwordspacing
V.~Apidopoulos, J.-F. Aujol, and C.~Dossal, ``Convergence rate of inertial
  forward--backward algorithm beyond nesterov's rule,'' \emph{Mathematical
  Programming}, vol. 180, no.~1, pp. 137--156, 2020. [Online]. Available:
  \url{https://doi.org/10.1007/s10107-018-1350-9}
\BIBentrySTDinterwordspacing

\bibitem{Moudafi:2006aa}
A.~Moudafi and P.-E. Maing{\'e}, ``Towards viscosity approximations of
  hierarchical fixed-point problems,'' \emph{Fixed Point Theory and
  Applications}, vol. 2006, pp. 1--10, 2006.

\bibitem{Mainge:2007aa}
P.-E. Maing{\'e} and A.~Moudafi, ``Strong convergence of an iterative method
  for hierarchical fixed-point problems,'' \emph{Pacific Journal of
  Optimization}, vol.~3, no.~3, pp. 529--538, 2007.

\bibitem{Marino:2011aa}
\BIBentryALTinterwordspacing
G.~Marino and H.-K. Xu, ``Explicit hierarchical fixed point approach to
  variational inequalities,'' \emph{Journal of Optimization Theory and
  Applications}, vol. 149, no.~1, pp. 61--78, 2011. [Online]. Available:
  \url{https://doi.org/10.1007/s10957-010-9775-1}
\BIBentrySTDinterwordspacing

\bibitem{BauCom16}
H.~H. Bauschke and P.~L. Combettes, \emph{Convex Analysis and Monotone Operator
  Theory in Hilbert Spaces}.\hskip 1em plus 0.5em minus 0.4em\relax Springer -
  CMS Books in Mathematics, 2016.

\bibitem{Alvarez:2001aa}
\BIBentryALTinterwordspacing
F.~Alvarez and H.~Attouch, ``An inertial proximal method for maximal monotone
  operators via discretization of a nonlinear oscillator with damping,''
  \emph{Set-Valued Analysis}, vol.~9, no.~1, pp. 3--11, 2001. [Online].
  Available: \url{https://doi.org/10.1023/A:1011253113155}
\BIBentrySTDinterwordspacing

\bibitem{Chambolle:2004aa}
\BIBentryALTinterwordspacing
A.~Chambolle, ``An algorithm for total variation minimization and
  applications,'' \emph{Journal of Mathematical Imaging and Vision}, vol.~20,
  no.~1, pp. 89--97, 2004. [Online]. Available:
  \url{https://doi.org/10.1023/B:JMIV.0000011325.36760.1e}
\BIBentrySTDinterwordspacing

\bibitem{Zhao:2009aa}
\BIBentryALTinterwordspacing
P.~Zhao, G.~Rocha, and B.~Yu, ``The composite absolute penalties family for
  grouped and hierarchical variable selection,'' \emph{The Annals of
  Statistics}, vol.~37, no.~6A, pp. 3468--3497, 12 2009. [Online]. Available:
  \url{https://doi.org/10.1214/07-AOS584}
\BIBentrySTDinterwordspacing

\bibitem{samadi2025improved}
S.~Samadi and F.~Yousefian, ``Improved guarantees for optimal nash equilibrium
  seeking and bilevel variational inequalities,'' \emph{SIAM Journal on
  Optimization}, vol.~35, no.~1, pp. 369--399, 2025.

\end{thebibliography}
\end{document}